\documentclass[11pt]{amsart}
\usepackage{amssymb,amsmath,amsthm,amsfonts,mathrsfs}
\usepackage[hmargin=3cm,vmargin=3.5cm]{geometry}
\usepackage{hyperref}
\usepackage{color}
\usepackage{graphicx}
\usepackage{amscd}  
\usepackage[shellescape]{gmp}
\usepackage{stmaryrd}  
\usepackage{mathtools}
\usepackage{comment} 
\usepackage{multirow} 
\usepackage{soul}
\usepackage[all]{xy}
\usepackage{relsize}

\usepackage{tikz}

\hypersetup{
    colorlinks=true,
    citecolor=brown,
    filecolor=brown,
    linkcolor=magenta,
    urlcolor=magenta
}

\newtheorem{theorem}{Theorem}[section]

\newtheorem{definition}[theorem]{Definition}

\newtheorem{proposition}[theorem]{Proposition}
\newtheorem{prop}[theorem]{Proposition}

\newtheorem{corollary}[theorem]{Corollary}

\let\oldtocsection=\tocsection
\let\oldtocsubsection=\tocsubsection
\renewcommand{\tocsection}[2]{\hspace{0em}\oldtocsection{#1}{#2}}
\renewcommand{\tocsubsection}[2]{\hspace{1em}\oldtocsubsection{#1}{#2}}

\usepackage{chngcntr}
\counterwithin{figure}{subsection}

\newcommand{\comp}{\mathsf{comp}}
\newcommand{\Cob}{\mathsf{Cob}}

\newcommand{\kk}{\mathbf{k}}

\newcommand{\lra}{\longrightarrow}

\newcommand{\mcD}{\mathcal{D}}

\newcommand{\cO}{\mathcal{O}}
\newcommand{\mcO}{\mathcal{O}}
\newcommand{\Orb}{\mathsf{Orb}}
\newcommand{\ddef}{\mathsf{def}}

\newcommand{\Z}{\mathbb{Z}}
\newcommand{\inn}{\mathsf{inn}}
\newcommand{\out}{\mathsf{out}}
\newcommand{\I}{\mathsf{I}}

\allowdisplaybreaks

\title[]{One-dimensional topological theories with defects and linear generating functions}
\author{Mee Seong Im}
 \address{Department of Mathematics, United States Naval Academy, Annapolis, MD 21402, USA}
 \email{\href{meeseongim@gmail.com}{meeseongim@gmail.com}}
\author{Paul Zimmer}
 \address{Department of Mathematics, United States Naval Academy, Annapolis, MD 21402, USA}
 \email{\href{m237146@usna.edu}{m237146@usna.edu}}

\date{\today}

\begin{document}

\maketitle

\begin{abstract}
We study the Gram determinant and construct bases of hom spaces for the one-dimensional topological theory of decorated unoriented one-dimensional cobordisms, as recently defined by Khovanov, 
when the pair of generating functions is linear.  
\end{abstract}

\section{Introduction}\label{section:intro}

Atiyah's celebrated topological quantum field theories framework~\cite{A} can be augmented by non-monoidal (in general) topological theories coming from the universal constructions as in~\cite{BHMV,FKNSWW,Kho20-univ-const,Kho20-1d-cob,KS}. In these constructions one starts with an evaluation of closed $n$-manifolds to elements of the ground field and associates state spaces to closed $(n-1)$-manifolds and maps between spaces to cobordisms between them. Khovanov~\cite{Kho20-1d-cob} have shown  that the resulting theories are interesting even in dimension one, if manifolds are enriched by defects, such as submanifolds or other decorations. In dimension one, defects are zero-dimensional submanifolds that may carry additional labels. Universal theories in this case  relate to noncommutative power series~\cite{Kho20-1d-cob}.  

Without the labels and in the one-dimensional case, one considers one-manifolds equipped with zero-dimensional submanifolds. They can be visualised by dots inside one-manifolds. A closed connected 1-manifold with such defects is diffeomorphic to a circle with some number $k$ of dots. Multiplicative evaluation function $\alpha$ as in~\cite{Kho20-1d-cob} is then determined by its values on such circles, one for each $k\ge 0$, and may be encoded by power series 
\begin{equation} \label{eq_Z_0}
Z_{\alpha}(T)=\sum_{k\ge 0} \gamma_k T^k , \ \ \
\gamma_k =\alpha(\mathrm{circle \ with }\ k\ \mathrm{dots})
\in \kk, 
\end{equation}
where $\kk$ is the ground field. As explained in~\cite{Kho20-1d-cob}, one can furthermore consider cobordisms that may have two types of boundary points, so that components may "end" in the middle of the cobordism, see Figure~\ref{fig2_1} below. This allows to view intervals as closed cobordisms, so that evaluation $\alpha$ acquires additional parameters, its values on an interval with $k$ dots. Then $\alpha$ is encoded by two generating functions, $Z_{\alpha}(T)$ in (\ref{eq_Z_0}) and 
\begin{equation} \label{eq_Z_1} 
Z_{1,\alpha}(T)=\sum_{k\ge 0} \beta_k T^k , \ \ \ \beta_k =\alpha(\mathrm{interval \ with }\ k\ \mathrm{dots})
\in \kk. 
\end{equation} 
Below, we relabel $Z_{\alpha}(T)$ to $Z_{0,\alpha}(T)$. 

To a generating function, there are  several categories of linearized cobordisms associated~\cite{KS,Kho20-1d-cob}. They have not been understood, and in this paper we treat a particular case when each of the two generating functions is linear. In this case we're able to compute the Gram determinant for the spanning set of cobordisms from the empty 0-manifold to $n$ points and show that it's nonzero. This gives a basis of hom spaces between any two objects $n,m$ of monoidal category $\Cob_{\alpha}$, the topological theory associated to $\alpha$, with easily computable composition of basis elements.

\section{Background}\label{section:background}

\subsection{Setup and the main result} 

We follow the framework of~\cite{Kho20-1d-cob} and consider the category of one-dimensional unoriented dotted cobordisms $\Cob_1$ with boundary. We work with the skeletal version of this category where objects are nonnegative integers $n\in \Z_+$. Morphisms from $n$ to $m$ are unoriented dotted one-dimensional manifolds $M$, with the boundary subdivided into the \emph{inner} and \emph{outer} parts, 
\begin{equation*}
    \partial M = (\partial M)_{\inn} \sqcup (\partial M)_{\out},  
\end{equation*}
and a fixed bijection between the outer part and the union of $n+m$ points, 
\begin{equation*}
    (\partial M)_{\out} \cong \{1,2,\dots, n+m\}. 
\end{equation*}
One thinks of $M$ as a cobordism from $n$ to $m$ ordered points, see Figure~\ref{fig2_1}. 

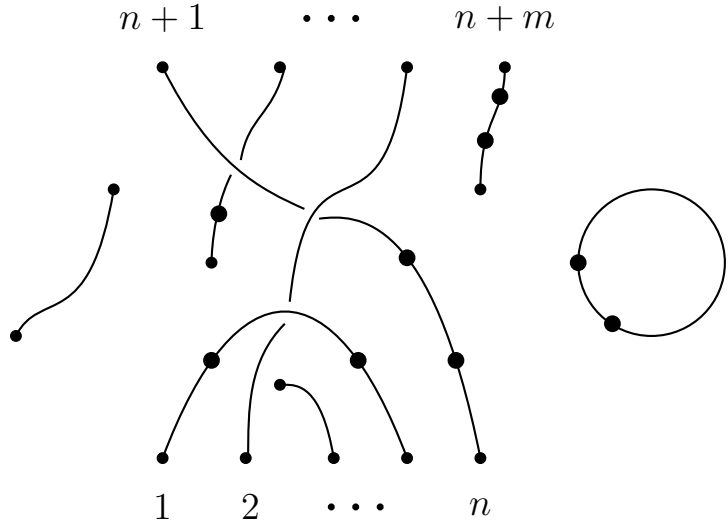
\begin{figure}
    \centering
\begin{tikzpicture}[scale=0.65]
\draw[thick] (0,0) .. controls (1.5,4) and (3.5,4) .. (5,0);
\draw[thick] (1.75,0) .. controls (1.75,1) and (1.75,2) .. (2.5,2.75);
\draw[thick] (2.6,3.2) .. controls (3,7) and (4.5,4) .. (5,8);
\draw[thick] (0,8) .. controls (1,6) and (2,5.5) .. (2.9,5.1);
\draw[thick] (3.2,4.9) .. controls (4,5) and (5.5,5) .. (6.5,0);
\draw[thick] (2.5,8) .. controls (2.25,7) and (1.75,7) .. (1.6,6.1);
\draw[thick] (1.4,5.8) .. controls (1,5) and (1,4) .. (1,4);
\draw[thick] (3.5,0) .. controls (3.25,1.5) and (2.75,1.5) .. (2.5,1.5);
\draw[thick] (7,8) .. controls (7,7) and (6.5,7) .. (6.5,5.5);
\draw[thick,fill] (.1,0) arc (0:360:1mm);
\draw[thick,fill] (1.8,0) arc (0:360:1mm);
\draw[thick,fill] (3.6,0) arc (0:360:1mm);
\draw[thick,fill] (5.1,0) arc (0:360:1mm);
\draw[thick,fill] (6.6,0) arc (0:360:1mm);
\draw[thick,fill] (.1,8) arc (0:360:1mm);
\draw[thick,fill] (2.5,8) arc (0:360:1mm);
\draw[thick,fill] (5.1,8) arc (0:360:1mm);
\draw[thick,fill] (7.1,8) arc (0:360:1mm);
\draw[thick,fill] (1.1,4) arc (0:360:1mm);
\draw[thick,fill] (6.6,5.5) arc (0:360:1mm);
\draw[thick,fill] (2.5,1.5) arc (0:360:1mm);
\draw[thick,fill] (1.3,5) arc (0:360:1.5mm);
\draw[thick,fill] (1.15,2) arc (0:360:1.5mm);
\draw[thick,fill] (4.15,2) arc (0:360:1.5mm);
\draw[thick,fill] (6.15,2) arc (0:360:1.5mm);
\draw[thick,fill] (5.15,4.1) arc (0:360:1.5mm);
\draw[thick,fill] (6.75,6.5) arc (0:360:1.5mm);
\draw[thick,fill] (7.05,7.4) arc (0:360:1.5mm);
\node at (0,-1) {\Large $1$};
\node at (1.8,-1) {\Large $2$};
\draw[thick,fill] (3.5,-1) arc (0:360:.5mm);
\draw[thick,fill] (4,-1) arc (0:360:.5mm);
\draw[thick,fill] (4.5,-1) arc (0:360:.5mm);
\node at (6.5,-1) {\Large $n$};
\node at (0,9) {\Large $n+1$};
\draw[thick,fill] (3,9) arc (0:360:.5mm);
\draw[thick,fill] (3.5,9) arc (0:360:.5mm);
\draw[thick,fill] (4,9) arc (0:360:.5mm);
\node at (7,9) {\Large $n+m$};
\draw[thick] (11.5,4) arc (0:360:1.5);
\draw[thick,fill] (8.65,4) arc (0:360:1.5mm);
\draw[thick,fill] (9.35,2.75) arc (0:360:1.5mm);
\draw[thick] (-3,2.5) .. controls (-2.5,3.5) and (-1.5,2.5) .. (-1,5.5);
\draw[thick,fill] (-2.9,2.5) arc (0:360:1mm);
\draw[thick,fill] (-.9,5.5) arc (0:360:1mm);
\end{tikzpicture}
    \caption{An unoriented dotted cobordism from $n$ to $m$ points. Defects are represented by large dots and boundary points of cobordisms are represented by small dots.}
    \label{fig2_1}
\end{figure}

These cobordisms are composed in the usual way. 
Cobordisms are considered up to rel outer boundary diffeomorphisms and can be described combinatorially. Each cobordism may have \emph{closed} and \emph{boundary} components. Closed components (with empty intersection with the outer boundary) are classified by their topological type (an interval or a circle) together with the number of dots on them (possibly none), see Figure~\ref{fig2_2}. Boundary components may have one or two outer boundary points and will be called \emph{$1$-arcs}, respectively, \emph{$2$-arcs}. Each boundary component may carry dots as well. 

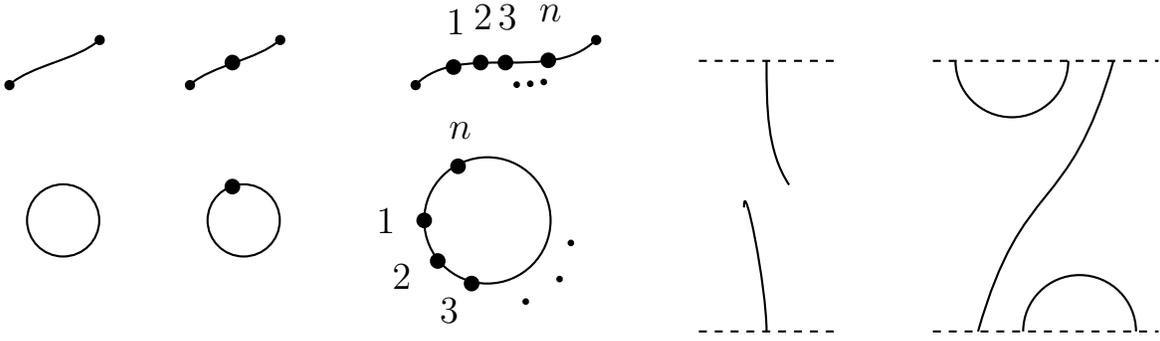
\begin{figure}
    \centering
\begin{tikzpicture}[scale=0.6]
\draw[thick] (0,0) arc (0:360:8mm);    
\draw[thick] (-2,3) .. controls (-1.5,3.45) and (-.5,3.55) .. (0,4);
\draw[thick,fill] (-1.9,3) arc (0:360:.9mm);
\draw[thick,fill] (.1,4) arc (0:360:.9mm);

\draw[thick] (4,0) arc (0:360:8mm); 
\draw[thick,fill] (3.1,.75) arc (0:360:1.5mm);

\draw[thick] (2,3) .. controls (2.5,3.45) and (3.5,3.55) .. (4,4);
\draw[thick,fill] (2.1,3) arc (0:360:.9mm);
\draw[thick,fill] (4.1,4) arc (0:360:.9mm);
\draw[thick,fill] (3.1,3.5) arc (0:360:1.5mm);

\draw[thick] (10,0) arc (0:360:14mm); 
\draw[thick,fill] (7.35,0) arc (0:360:1.5mm);
\node at (6.35,0) {\Large $1$};
\draw[thick,fill] (7.65,-.9) arc (0:360:1.5mm);
\node at (6.7,-1.25) {\Large $2$};
\draw[thick,fill] (8.4,-1.4) arc (0:360:1.5mm);
\node at (7.75,-2) {\Large $3$};

\draw[thick,fill] (9.5,-1.8) arc (0:360:.5mm);
\draw[thick,fill] (10.25,-1.3) arc (0:360:.5mm);
\draw[thick,fill] (10.5,-.5) arc (0:360:.5mm);

\draw[thick,fill] (8.1,1.2) arc (0:360:1.5mm);
\node at (8,2) {\Large $n$};

\draw[thick] (7,3) .. controls (8,4) and (10,3) .. (11,4);

\draw[thick,fill] (7.1,3) arc (0:360:.9mm);
\draw[thick,fill] (11.1,4) arc (0:360:.9mm);

\draw[thick,fill] (8,3.4) arc (0:360:1.5mm);
\node at (7.9,4.4) {\Large $1$};
\draw[thick,fill] (8.6,3.5) arc (0:360:1.5mm);
\node at (8.5,4.5) {\Large $2$};
\draw[thick,fill] (9.15,3.5) arc (0:360:1.5mm);
\node at (9.05,4.5) {\Large $3$};

\draw[thick,fill] (9.3,3.01) arc (0:360:.5mm);
\draw[thick,fill] (9.6,3.03) arc (0:360:.5mm);
\draw[thick,fill] (9.9,3.07) arc (0:360:.5mm);

\draw[thick,fill] (10.1,3.55) arc (0:360:1.5mm);
\node at (10,4.6) {\Large $n$};
\end{tikzpicture}
\qquad  \quad
\begin{tikzpicture}[scale=0.6]
\draw[thick,dashed] (0,0) -- (3,0);
\draw[thick,dashed] (0,6) -- (3,6);
\draw[thick] (1.5,0) .. controls (1.5,1) and (1,3.5) .. (1,2.75); 
\draw[thick] (1.5,6) .. controls (1.5,5) and (1.5,4) .. (2,3.25);
\end{tikzpicture}
\qquad  \quad 
\begin{tikzpicture}[scale=0.6]
\draw[thick,dashed] (0,0) -- (5,0);
\draw[thick,dashed] (0,6) -- (5,6);
\draw[thick] (1,0) .. controls (2,3.5) and (3,2.5) .. (4,6);
\draw[thick] (4.5,0) arc (0:180:1.25);
\draw[thick] (.5,6) arc (180:360:1.25);
\end{tikzpicture}
    \caption{On the left: possible floating connected components of a cobordism include an arc with $n\ge 0$ dots and a circle with $n\ge 0$ dots. On the right: five types of components which contain one or two outer boundary points, including two types of 1-arcs and three types of 2-arcs.}
    \label{fig2_2}
\end{figure}

An evaluation function or series $\alpha$ is a map from the set of diffeomorphism classes of closed connected diagrams to the ground field $\kk$. It can be encoded by two power series
\begin{equation} 
Z_{0,\alpha}(T)=\sum_{n\ge 0} \beta_n T^n, \qquad 
Z_{1,\alpha}(T)=\sum_{n\ge 0} \gamma_n T^n, 
\end{equation} 
where $\beta_n$, respectively $\gamma_n$, is the value of $\alpha$ on an interval with $n$ dots, respectively on a circle with $n$ dots, $n\ge 0$, see Figure~\ref{fig2_4_2}. 

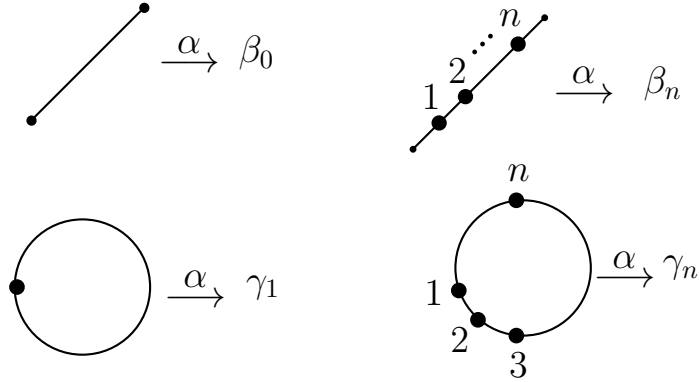
\begin{figure}
    \centering
\begin{tikzpicture}[scale=0.6]
\draw[thick] (-.5,-.5) -- (2,2);
\draw[thick,fill] (-.4,-.5) arc (0:360:.9mm);
\draw[thick,fill] (2.1,2) arc (0:360:.9mm);
\node at (3,1) {$\stackrel{\mbox{\Large $\alpha$}}{\mbox{\Large $\longrightarrow$}}$};
\node at (4.5,1) {\Large $\beta_0$};
\node at (0,-1) {};
\end{tikzpicture}
\qquad \qquad 
\begin{tikzpicture}[scale=0.35]
\draw[thick] (0,0) -- (5,5); 
\draw[thick,fill] (0.1,0) arc (0:360:.9mm);
\draw[thick,fill] (5.1,5) arc (0:360:.9mm);
\draw[thick,fill] (1.25,1) arc (0:360:2.5mm);
\node at (.7,2) {\Large $1$};
\draw[thick,fill] (2.25,2) arc (0:360:2.5mm);
\node at (1.7,3) {\Large $2$};
\draw[thick,fill] (4.25,4) arc (0:360:2.5mm);
\node at (3.7,5) {\Large $n$};
\draw[thick,fill] (2.4,3.7) arc (0:360:.5mm);
\draw[thick,fill] (2.7,4) arc (0:360:.5mm);
\draw[thick,fill] (3,4.3) arc (0:360:.5mm);
\node at (6.5,2.5) {${\stackrel{\mbox{\Large $\alpha$}}{\mbox{\Large $\longrightarrow$}}}$};
\node at (9.5,2.5) {\Large $\beta_n$};
\end{tikzpicture} \\ 
\begin{tikzpicture}[scale=0.6]
\draw[thick] (0,0) arc (0:360:1.5);
\draw[thick,fill] (-2.8,0) arc (0:360:1.5mm);
\node at (1,0) {$\stackrel{\mbox{\Large $\alpha$}}{\mbox{\Large $\longrightarrow$}}$};
\node at (2.5,0) {\Large $\gamma_1$};
\node at (0,-2) {};
\end{tikzpicture}
\qquad \qquad 
\begin{tikzpicture}[scale=0.3]
\draw[thick] (0,0) arc (0:360:3);
\draw[thick,fill] (-5.55,-1) arc (0:360:3mm);
\node at (-6.95,-1.1) {\Large $1$};
\draw[thick,fill] (-4.7,-2.3) arc (0:360:3mm);
\node at (-5.8,-3.2) {\Large $2$};
\draw[thick,fill] (-3,-3) arc (0:360:3mm);
\node at (-3.2,-4.3) {\Large $3$};
\draw[thick,fill] (-3,3) arc (0:360:3mm);
\node at (-3.2,4.2) {\Large $n$};
\node at (1.5,0) {$\stackrel{\mbox{\Large $\alpha$}}{\mbox{\Large $\longrightarrow$}}$};
\node at (4,0) {\Large $\gamma_n$};
\end{tikzpicture}
    \caption{ Evaluation function $\alpha$. }

    \label{fig2_4_2}
\end{figure}

To an evaluation $\alpha$ one can associate several  categories and functors between them as in~\cite{KS,Kho20-1d-cob}. 
To $n$ points there's also associated a vector space $A(n)$. One considers a vector space $V(n)$ with a basis of all cobordisms $x$ as above from $0$ to $n$. Denote by $\overline{x}$ the reflection of $x$, a cobordism from $n$ to $0$, see Figure~\ref{fig2_4_1} top.

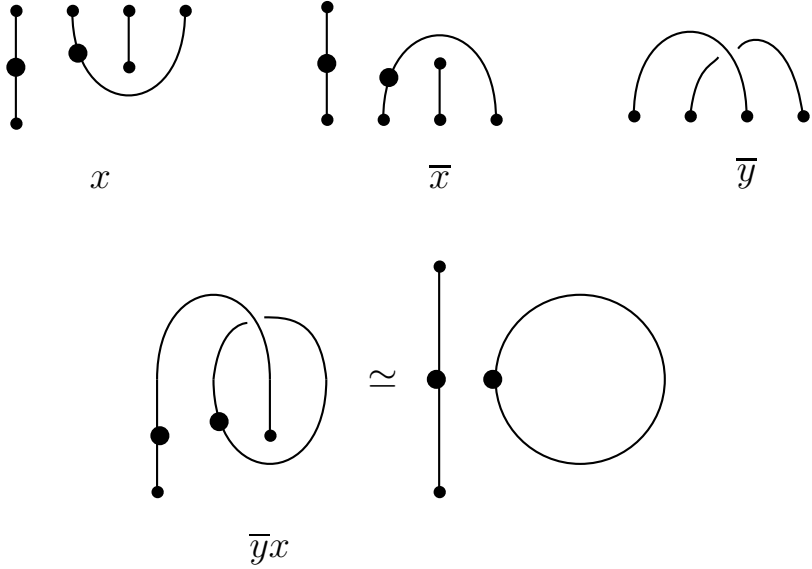
\begin{figure}
    \centering
\begin{tikzpicture}[scale=0.75]
\draw[thick] (0.5,0) -- (0.5,1);
\draw[thick] (0.5,1) -- (0.5,2);
\draw[thick,fill] (.6,0) arc (0:360:.9mm);
\draw[thick,fill] (.65,1) arc (0:360:1.5mm);
\draw[thick,fill] (.6,2) arc (0:360:.9mm);
\draw[thick] (1.5,2) .. controls (1.5,0) and (3.5,0) .. (3.5,2);
\draw[thick,fill] (1.6,2) arc (0:360:.9mm);
\draw[thick,fill] (3.6,2) arc (0:360:.9mm);
\draw[thick] (2.5,2) -- (2.5,1);
\draw[thick,fill] (2.6,2) arc (0:360:.9mm);
\draw[thick,fill] (2.6,1) arc (0:360:.9mm);
\draw[thick,fill] (1.75,1.25) arc (0:360:1.5mm);
\node at (2,-1) {\Large $x$};
\end{tikzpicture}
\qquad 
\qquad 
\begin{tikzpicture}[scale=0.75]
\draw[thick] (0,0) -- (0,2);
\draw[thick,fill] (.15,1) arc (0:360:1.5mm);
\draw[thick,fill] (.1,0) arc (0:360:.9mm);
\draw[thick,fill] (.1,2) arc (0:360:.9mm);
\node at (2,-1) {\Large $\overline{x}$};
\draw[thick] (1,0) .. controls (1,2) and (3,2) .. (3,0);
\draw[thick,fill] (1.1,0) arc (0:360:.9mm);
\draw[thick,fill] (3.1,0) arc (0:360:.9mm);
\draw[thick] (2,0) -- (2,1);
\draw[thick,fill] (2.1,1) arc (0:360:.9mm);
\draw[thick,fill] (2.1,0) arc (0:360:.9mm);
\draw[thick,fill] (1.25,.75) arc (0:360:1.5mm);
\end{tikzpicture} 
\qquad \qquad 
\begin{tikzpicture}[scale=0.75]
\draw[thick] (0,0) .. controls (0,2) and (2,2) .. (2,0);
\draw[thick] (1,0) .. controls (1.1,.9) and (1.4,.9) .. (1.5,1.05);
\draw[thick] (1.85,1.2) .. controls (2.05,1.5) and (2.8,1.5) .. (3,0);
\draw[thick,fill] (0.1,0) arc (0:360:.9mm);
\draw[thick,fill] (1.1,0) arc (0:360:.9mm);
\draw[thick,fill] (2.1,0) arc (0:360:.9mm);
\draw[thick,fill] (3.1,0) arc (0:360:.9mm);
\node at (2,-1) {\Large $\overline{y}$};
\end{tikzpicture} \\ 
\begin{tikzpicture}[scale=0.75]
\draw[thick] (0,0) -- (0,1);
\draw[thick,fill] (.1,0) arc (0:360:.9mm);
\draw[thick] (0,1) -- (0,2);
\draw[thick,fill] (.2,1) arc (0:360:1.5mm);
\draw[thick] (1,2) .. controls (1,0) and (3,0) .. (3,2);
\draw[thick] (2,2) -- (2,1);
\draw[thick,fill] (2.1,1) arc (0:360:.9mm);
\draw[thick,fill] (1.25,1.25) arc (0:360:1.5mm);

\draw[thick] (0,2) .. controls (0,4) and (2,4) .. (2,2);
\draw[thick] (1,2) .. controls (1.1,2.9) and (1.5,3.0) .. (1.6,3.0);
\draw[thick] (1.9,3.1) .. controls (2.3,3.1) and (2.9,3.1) .. (3,2);
\node at (2,-1) {\Large $\overline{y}x$};
\node at (4,2) {\Large $\simeq$};
\draw[thick] (5,0) -- (5,4);
\draw[thick,fill] (5.1,4) arc (0:360:.9mm);
\draw[thick,fill] (5.1,0) arc (0:360:.9mm);
\draw[thick] (9,2) arc (0:360:1.5);
\draw[thick,fill] (5.1,2) arc (0:360:1.5mm);
\draw[thick,fill] (6.1,2) arc (0:360:1.5mm);
\node at (5,5) {};
\end{tikzpicture}
    \caption{ An example of a cobordism $x$ from $0$ to $4$, reflections $\overline{x},\overline{y}$ and coupling $\overline{y}x$. } 
    \label{fig2_4_1}
\end{figure}

For two cobordisms $x,y$ from $0$ to $n$, the composition $\overline{y}x$ is a cobordism from $0$ to $0$ (a closed cobordism) and can be evaluated via $\alpha$, see an example on Figure~\ref{fig2_4_1}.

This induces a bilinear form on $V(n)$, 
\begin{equation}
   (\:\:,\:\:) \ : \  V(n)\times V(n) \lra \kk,  \hspace{4mm} 
   (x,y) = \alpha(\overline{y}x). 
\end{equation}
Define $A(n)$ as the quotient of $V(n)$ by the kernel of the bilinear form, 
\begin{equation*}
    A(n) \ := \ V(n)/\mathrm{ker}((\:\cdot , \cdot \:)). 
\end{equation*}
From vector spaces $A(n)$ one can build monoidal category $\Cob_{\alpha}$ as  in~\cite{Kho20-1d-cob} by identifying homs from $n$ to $m$ in the latter category with $A(n+m)$ and forming suitable compositions of morphisms.

In this manuscript, we consider a very special evaluation $\alpha$; it is nonzero only on four diffeomorphism classes of decorated connected 1-manifolds and  
\begin{itemize}
    \item takes values $\beta_0,\beta_1$ on an interval with no dots and one dot, respectively;
    \item takes values $\gamma_0,\gamma_1$ on a circle with no dots and one dot, respectively; 
    \item evaluates any connected 1-manifold with at least two dots to $0$; 
\end{itemize}
see Figure~\ref{fig2_7}.

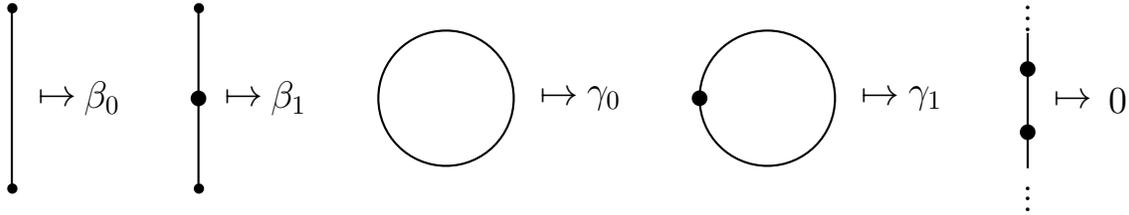
\begin{figure}
    \centering
\begin{tikzpicture}[scale=0.6]
\draw[thick] (0,0) -- (0,4);
\draw[thick,fill] (0.1,4) arc (0:360:.9mm);
\draw[thick,fill] (0.1,0) arc (0:360:.9mm);
\node at (1,2) {\Large $\mapsto$};
\node at (2,2) {\Large $\beta_0$};
\node at (1,-.5) {};
\end{tikzpicture}
\quad  \:\:
\begin{tikzpicture}[scale=0.6]
\draw[thick] (0,0) -- (0,4);
\draw[thick,fill] (0.1,4) arc (0:360:.9mm);
\draw[thick,fill] (0.1,0) arc (0:360:.9mm);
\draw[thick,fill] (0.15,2) arc (0:360:1.5mm);
\node at (1,2) {\Large $\mapsto$};
\node at (2,2) {\Large $\beta_1$};
\node at (1,-.5) {};
\end{tikzpicture}
\quad  \:\: 
\begin{tikzpicture}[scale=0.6]
\draw[thick] (0,0) arc (0:360:1.5);  
\node at (1,0) {\Large $\mapsto$};
\node at (2,0) {\Large $\gamma_0$};
\node at (1,-2.5) {};
\end{tikzpicture}
\quad  \:\: 
\begin{tikzpicture}[scale=0.6]
\draw[thick] (0,0) arc (0:360:1.5);  
\draw[thick,fill] (-2.85,0) arc (0:360:1.5mm);
\node at (1,0) {\Large $\mapsto$};
\node at (2,0) {\Large $\gamma_1$}; 
\node at (1,-2.5) {};
\end{tikzpicture}
\quad \:\:
\begin{tikzpicture}[scale=0.6]
\draw[thick] (0,0.5) -- (0,3.5);
\draw[thick,fill] (0.15,1.3) arc (0:360:1.5mm);
\draw[thick,fill] (0.15,2.7) arc (0:360:1.5mm);
\node at (1,2) {\Large $\mapsto$};
\node at (2,2) {\Large $0$};    
\node at (0,0) {\Large $\vdots$};    
\node at (0,4) {\Large $\vdots$};    
\end{tikzpicture}
    \caption{Closed diagrams under the evaluation map $\alpha$.}
    \label{fig2_7}
\end{figure}

We study vector spaces $A(n)$ and associated categories for this special evaluation $\alpha$, for which both $Z_{0,\alpha},Z_{1,\alpha}$  are linear  functions, 
\begin{equation}
    Z_0(\alpha) = \beta_0 + \beta_1 T , \qquad 
    Z_1(\alpha) = \gamma_0+ \gamma_1 T, \qquad  \beta_0,\beta_1,\gamma_0,\gamma_1\in \kk. 
\end{equation}

The evaluation $\alpha$ is zero once a closed cobordism has a component with at least two dots, see Figure~\ref{fig2_7} right, since quadratic and all higher order terms of $Z_{0,\alpha},Z_{1,\alpha}$ are $0$. This means there's a relation that two dots on the line evaluate to $0$, both in the skein category and in $\Cob_{\alpha}$. In particular, in $\Cob_{\alpha}$ a morphism can be reduced to a linear combination of cobordisms with no closed components and with at most one dot on each component. Likewise, vector space $A(n)$ has a spanning set of such cobordisms (from $0$ to $n$). We call these cobordisms \emph{$n$-diagrams} and denote their set by $\mcD_n$. A connected component of an $n$-diagram can be one of the four types, see Figure~\ref{fig2_5_1}: 
\begin{enumerate}
    \item a dotless arc, 
    \item arc with one dot, 
    \item a dotless cup, 
    \item a cup with one dot. 
\end{enumerate}
Types (1) and (2) may also be called 1-arcs, types (3) and (4) may be called 2-arcs. 

\begin{figure}
    \centering
\begin{tikzpicture}[scale=0.6]
\draw[thick,dashed] (0,0) -- (4,0); 
\draw[thick] (2,0) -- (2,-2);
\draw[thick,fill] (2.1,-2) arc (0:360:.9mm);
\node at (2,-3) {\Large $(1)$};
\draw[thick,dashed] (6,0) -- (10,0); 
\draw[thick,fill] (8.15,-1) arc (0:360:1.5mm);
\draw[thick,fill] (8.1,-2) arc (0:360:.9mm);
\draw[thick] (8,0) -- (8,-2);
\node at (8,-3) {\Large $(2)$};
\draw[thick,dashed] (12,0) -- (16,0); 
\draw[thick] (12.5,0) arc (180:360:1.5);
\node at (14,-3) {\Large $(3)$};
\draw[thick,dashed] (18,0) -- (22,0); 
\draw[thick] (18.5,0) arc (180:360:1.5);
\draw[thick,fill] (20.15,-1.5) arc (0:360:1.5mm);
\node at (20,-3) {\Large $(4)$};
\end{tikzpicture}
    \caption{Four types of components of an $n$-diagram.  Types (1) and (2) are 1-arcs, types (3) and (4) are 2-arcs, see earlier.}
    \label{fig2_5_1}
\end{figure}
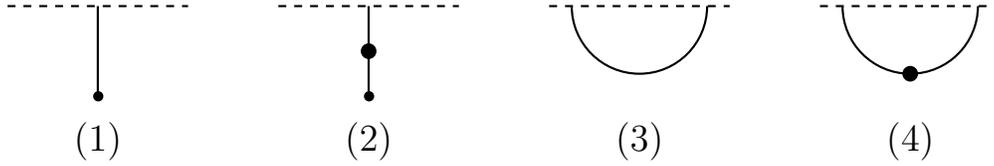

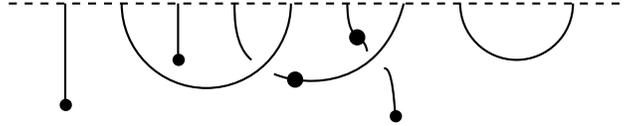
\begin{figure}
    \centering
\begin{tikzpicture}[scale=0.75]
\draw[thick,dashed] (0,0) -- (11,0);
\draw[thick] (1,0) -- (1,-1.8);
\draw[thick,fill] (1.1,-1.8) arc (0:360:.9mm);
\draw[thick] (2,0) arc (180:360:1.5);

\draw[thick] (3,0) -- (3,-1);
\draw[thick,fill] (3.1,-1) arc (0:360:.9mm);

\draw[thick] (4,0) .. controls (4,-.25) and (4,-.75) .. (4.3,-1);
\draw[thick] (4.7,-1.25) .. controls (5.3,-1.5) and (6.6,-1.5) .. (7,0);
\draw[thick,fill] (5.2,-1.35) arc (0:360:1.25mm);

\draw[thick] (6,0) .. controls (6,-.75) and (6.3,-.5) .. (6.35,-.85);
\draw[thick,fill] (6.3,-.6) arc (0:360:1.25mm);

\draw[thick] (6.65,-1.15) .. controls (6.70,-1.15) and (6.8,-1.15) .. (6.85,-2);
\draw[thick,fill] (6.95,-2) arc (0:360:.9mm);

\draw[thick] (8,0) arc (180:360:1);

\end{tikzpicture}
    \caption{Example of a diagram in the orbit $\cO(2,1,2,1)$.}
    \label{fig2_5_2}
\end{figure}

The symmetric group $S_n$ acts on the set of $n$-diagrams by permuting the endpoints. Denote the set of orbits of the  action by $\Orb_n$. 
For an $n$-diagram $X\in \mcD_n$ let 
\begin{itemize}
\item $a$ be the number of dotless strands, 
\item $b$ be the number of strands with a single dot, 
\item $c$ be the number of dotless cups, and 
\item $d$ be the number of cups with a single dot in a diagram. 
\end{itemize} 
We have 
\begin{equation*}
    a+b+2c+2d = n, 
\end{equation*}
and the quadruple $(a,b,c,d)$ is a complete invariant of an orbit of $S_n$ on $\mcD_n$. Figure~\ref{fig2_5_2} shows an 
example of a 9-diagram with parameters $(2,1,2,1)$. 

We define $\cO(a,b,c,d)=S_n X \in \Orb_n$ as the orbit which consist of all $n$-diagrams with  the parameters $(a,b,c,d)$. Then 
\begin{equation*}
    \mcD_n = \coprod_{a,b,c,d} \cO(a,b,c,d), \hspace{4mm} 
    a+b+2c+2d=n,  \hspace{4mm} 
    a,b,c,d\ge 0.
\end{equation*}

For a diagram $X\in \cO(a,b,c,d)$, define 
\begin{itemize}
    \item $\ddef(X)=b+d$, the number of defect points in the diagram $X$,
     \item $\comp(X)=a+b+c+d$, the total number of connected components of $X$. 
\end{itemize}

\vspace{0.1in} 

Diagrams in $\mcD_n$ span $A(n)$. To see whether they constitute a basis of $A(n)$ we need to form the Gram matrix by pairing these diagrams and see whether it's nondenegenerate. 

Recall that the reflection $\overline{X}$ of an $n$-diagram $X$ is given by reflecting $X$ in a horizontal line, so it becomes a cobordism from $n$ to $0$. Given two $n$-diagrams $X,Y$, we can compose $\overline{Y}$ and $X$ to get a closed diagram $\overline{Y}X$. It can then be evaluated via $\alpha$. Thus, we define
\begin{equation}
    \alpha(X,Y) \ := \ \alpha(\overline{Y}X) \in \kk.
\end{equation}
An example can be found in Figure~\ref{fig2_4_1}, for $n=4$ and diagrams $x$ and $y$ with parameters $(1,1,0,1)$ and $(0,0,2,0)$, respectively. 
Evaluation $\alpha(\overline{y}x) = \beta_1\gamma_1$. 

\vspace{0.1in} 

Picking an order on the set $\mcD_n$ of $n$-diagrams gives us the Gram matrix $G$. 
We now explain how to compute the determinant of $G$. It turns out that, after suitable permutations of rows and columns, $G$
becomes lower-triangular. To see this, we first define the notion of the dual of an $n$-diagram.

\begin{definition}
\label{defn:dual-diagram}
The \textbf{dual} $\iota(X)$ of a diagram $X\in \mcD_n$ is the diagram with the same arcs and cups as $X$, with the dots removed from single-dotted components and a single dot added to each dotless component.  
\end{definition}

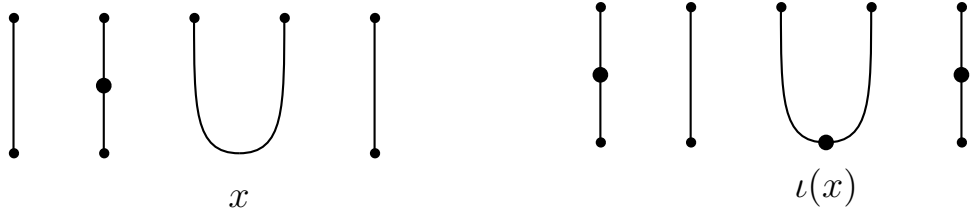
\begin{figure}
    \centering
\begin{tikzpicture}[scale=0.6]
\draw[thick] (-3,0) -- (-3,3);
\draw[thick,fill] (-2.9,3) arc (0:360:.9mm);
\draw[thick,fill] (-2.9,0) arc (0:360:.9mm);
\draw[thick] (-1,0) -- (-1,3);
\draw[thick,fill] (-.9,3) arc (0:360:.9mm);
\draw[thick,fill] (-.9,0) arc (0:360:.9mm);
\draw[thick,fill] (-.85,1.5) arc (0:360:1.5mm);
\draw[thick] (1,3) .. controls (1,1) and (1,0) .. (2,0);
\draw[thick,fill] (1.1,3) arc (0:360:.9mm);
\draw[thick,fill] (3.1,3) arc (0:360:.9mm);
\draw[thick] (3,3) .. controls (3,1) and (3,0) .. (2,0);
\draw[thick] (5,0) -- (5,3);
\draw[thick,fill] (5.1,3) arc (0:360:.9mm);
\draw[thick,fill] (5.1,0) arc (0:360:.9mm);
\node at (2,-1) {\Large $x$};
\end{tikzpicture}
\qquad \qquad \qquad \quad 
\begin{tikzpicture}[scale=0.6]
\draw[thick,fill] (-2.9,3) arc (0:360:.9mm);
\draw[thick,fill] (-2.9,0) arc (0:360:.9mm);
\draw[thick] (-3,0) -- (-3,3);
\draw[thick,fill] (-2.85,1.5) arc (0:360:1.5mm);
\draw[thick] (-1,0) -- (-1,3);
\draw[thick,fill] (-.9,3) arc (0:360:.9mm);
\draw[thick,fill] (-.9,0) arc (0:360:.9mm);
\draw[thick] (1,3) .. controls (1,1) and (1,0) .. (2,0);
\draw[thick,fill] (3.1,3) arc (0:360:.9mm);
\draw[thick,fill] (1.1,3) arc (0:360:.9mm);
\draw[thick,fill] (2.15,0) arc (0:360:1.5mm);
\draw[thick] (3,3) .. controls (3,1) and (3,0) .. (2,0);
\draw[thick] (5,0) -- (5,3);
\draw[thick,fill] (5.1,3) arc (0:360:.9mm);
\draw[thick,fill] (5.1,0) arc (0:360:.9mm);
\draw[thick,fill] (5.15,1.5) arc (0:360:1.5mm);
\node at (2,-1) {\Large $\iota(x)$};
\end{tikzpicture}
    \caption{Example of a diagram $x$ and its dual $\iota(x)$ when $n=5$. }
    \label{fig2_3}
\end{figure}
 

An example can be found in Figure~\ref{fig2_3}. 
Map $\iota: \mcD_n\lra \mcD_n$ is an involution on the set of $n$-diagrams. It descends to an involution on the set of $S_n$-orbits, also denoted $\iota$. 

It is convenient to twist the pairing $\alpha(\:\:,\:\:)$ by $\iota$ and define $\alpha_{\iota}(X,Y):=\alpha(X,\iota(Y))=\alpha(\overline{\iota(Y)}X)$. An example of twisted pairing is computed in Figure~\ref{fig2_5}.

Both pairings are maps $\alpha,\alpha_{\iota}: \mcD_n\times \mcD_n \lra \kk$.


\subsection{Understanding the Gram matrix}
\label{subsection:gluing-diagrams-Gram-matrix}

Up to a possible sign, instead of computing determinant of the Gram matrix $G$ we can compute the determinant of the twisted matrix $G_{\iota}$ with the entry on the intersection of $X$-row and $Y$-column  $\alpha_{\iota}(X,Y)=\alpha(\overline{\iota(Y)}X).$

\vspace{0.1in}

If one of the components of closed diagram $\overline{\iota(Y)}X$  contains at least two dots, the evaluation is zero. Otherwise it is  $\beta_0^{k_0}\beta_1^{k_1}\gamma_0^{k_2}\gamma_1^{k_3}$, where $k_0,k_1,k_2,k_3$ are the number of free (dotless) arcs, arcs with a dot, free circles, and circles with a dot, respectively, see Figure~\ref{fig2_5} for an example. 


\begin{figure}
    \centering
\begin{tikzpicture}[scale=0.5]
\node at (-5,0) {\Large $\alpha_{\iota}$};
\draw[thick] (-4,1.75) .. controls (-4.5,.75) and (-4.5,-.75) .. (-4,-1.75);
\draw[thick] (-3,-1.5) -- (-3,1.5);
\draw[thick,fill] (-2.9,1.5) arc (0:360:.9mm);
\draw[thick,fill] (-2.9,-1.5) arc (0:360:.9mm);
\draw[thick,fill] (-2.82,0) arc (0:360:1.75mm);
\draw[thick] (-1,-1.5) -- (-1,1.5);
\draw[thick,fill] (-.9,1.5) arc (0:360:.9mm);
\draw[thick,fill] (-.9,-1.5) arc (0:360:.9mm);
\draw[thick] (1,1.5) .. controls (1,-0.5) and (1,-1.5) .. (2,-1.5);
\draw[thick,fill] (1.1,1.5) arc (0:360:.9mm);
\draw[thick,fill] (3.1,1.5) arc (0:360:.9mm);
\draw[thick] (3,1.5) .. controls (3,-0.5) and (3,-1.5) .. (2,-1.5);
\node at (4,-1.5) {\Large $\Huge{,}$};
\draw[thick] (5,-1.5) -- (5,1.5);
\draw[thick,fill] (5.1,1.5) arc (0:360:.9mm);
\draw[thick,fill] (5.1,-1.5) arc (0:360:.9mm);
\draw[thick] (7,1.5) .. controls (7,-0.5) and (7,-1.5) .. (8,-1.5);
\draw[thick,fill] (7.1,1.5) arc (0:360:.9mm);
\draw[thick,fill] (9.1,1.5) arc (0:360:.9mm);
\draw[thick] (9,1.5) .. controls (9,-0.5) and (9,-1.5) .. (8,-1.5);
\draw[thick,fill] (8.2,-1.5) arc (0:360:1.75mm);

\draw[thick] (11,-1.5) -- (11,1.5);
\draw[thick,fill] (11.1,1.5) arc (0:360:.9mm);
\draw[thick,fill] (11.1,-1.5) arc (0:360:.9mm);
\draw[thick] (12,1.75) .. controls (12.5,.75) and (12.5,-.75) .. (12,-1.75);
\node at (13,0) {\Large $=$};
\node at (14,0) {\Large $\alpha$};
\draw[thick] (15,1.75) .. controls (14.5,.75) and (14.5,-.75) .. (15,-1.75);
\draw[thick] (16,-1.5) -- (16,-0.15);
\draw[thick,fill] (16.1,-1.5) arc (0:360:.9mm);
\draw[thick,fill] (16.1,-.15) arc (0:360:.9mm);
\draw[thick,fill] (16.2,-.85) arc (0:360:1.75mm);

\draw[thick] (18,-1.5) -- (18,-.15);
\draw[thick,fill] (18.1,-.15) arc (0:360:.9mm);
\draw[thick,fill] (18.1,-1.5) arc (0:360:.9mm);

\draw[thick] (20,-.15) .. controls (20,-0.5) and (20,-1.5) .. (21,-1.5);
\draw[thick,fill] (20.1,-.15) arc (0:360:.9mm);
\draw[thick,fill] (22.1,-.15) arc (0:360:.9mm);
\draw[thick] (22,-.15) .. controls (22,-0.5) and (22,-1.5) .. (21,-1.5);

\draw[thick] (16,1.5) -- (16,.11);
\draw[thick,fill] (16.1,1.5) arc (0:360:.9mm);
\draw[thick,fill] (16.1,.11) arc (0:360:.9mm);
\draw[thick,fill] (16.2,.85) arc (0:360:1.75mm);
 
\draw[thick] (18,.15) .. controls (18,0.5) and (18,1.5) .. (19,1.5);
\draw[thick,fill] (18.1,.15) arc (0:360:.9mm);
\draw[thick,fill] (20.1,.15) arc (0:360:.9mm);
\draw[thick] (20,.15) .. controls (20,0.5) and (20,1.5) .. (19,1.5);

\draw[thick] (22,1.5) -- (22,.15);
\draw[thick,fill] (22.1,1.5) arc (0:360:.9mm);
\draw[thick,fill] (22.1,.15) arc (0:360:.9mm);
\draw[thick,fill] (22.2,.85) arc (0:360:1.75mm);
\draw[thick] (23,1.75) .. controls (23.5,.75) and (23.5,-.75) .. (23,-1.75); 

\node at (13,-4) {\Large $=$};
\node at (14,-4) {\Large $\alpha$};
\draw[thick] (15,-2.25) .. controls (14.5,-3.25) and (14.5,-4.75) .. (15,-5.75);

\draw[thick] (16,-5.5) -- (16,-4);
\draw[thick,fill] (16.1,-5.5) arc (0:360:.9mm);
\draw[thick,fill] (16.1,-2.5) arc (0:360:.9mm);
\draw[thick,fill] (16.2,-4.75) arc (0:360:1.75mm);

\draw[thick] (16,-4) -- (16,-2.5);
\draw[thick,fill] (16.2,-3.25) arc (0:360:1.75mm);

\draw[thick] (18,-5.5) -- (18,-4);

\draw[thick,fill] (18.1,-5.5) arc (0:360:.9mm);

\draw[thick] (18,-4) .. controls (18,-3.5) and (18,-2.5) .. (19,-2.5);

\draw[thick] (20,-4) .. controls (20,-3.5) and (20,-2.5) .. (19,-2.5);

\draw[thick] (20,-4) .. controls (20,-4.5) and (20,-5.5) .. (21,-5.5);
 
\draw[thick] (22,-4) .. controls (22,-4.5) and (22,-5.5) .. (21,-5.5);

\draw[thick] (22,-2.5) -- (22,-4);
\draw[thick,fill] (22.2,-3.25) arc (0:360:1.75mm);
\draw[thick,fill] (22.1,-2.5) arc (0:360:.9mm);

\draw[thick] (23,-2.25) .. controls (23.5,-3.25) and (23.5,-4.75) .. (23,-5.75);

\node at (13,-7) {\Large $=$}; 
\node at (14,-7) {\Large $0$.};
\end{tikzpicture}
    \caption{An example of evaluation $\alpha_{\iota}(X,Y)=\alpha(\overline{\iota(Y)}X)$, with  $X\in \mathcal{O}(1,1,1,0)$ and $Y\in \mathcal{O}(2,0,0,1)$. }
    \label{fig2_5}
\end{figure}

\begin{definition}
\label{defn:strict-total-ordering}
Let $X\in \mathcal{O}(a,b,c,d)$ and $Y \in \mathcal{O}(a',b',c',d')$. We define a \textbf{strict total ordering} on the diagrams, and write  $X < Y$, if 
\begin{enumerate}
    \item $\ddef(X) > \ddef(Y)$, or 
    \item $\ddef(X) = \ddef(Y)$ and $\comp(X)<\comp(Y)$, or 
    \item $\ddef(X)=\ddef(Y)$ and $\comp(X) = \comp(Y)$, and $d> d'$. 
\end{enumerate}
\end{definition}
Note that, if $\ddef(X)=\ddef(Y)$, $\comp(X)=\comp(Y)$, and $d=d'$, then $X= Y$. 
Orbits with more defects (dots) appear earlier in the order than orbits with fewer defects, according to our definition. 

\vspace{0.1in} 

For each orbit $\mcO(a,b,c,d)\in \Orb_n$ we choose an arbitrary total ordering of its elements. 
Next, order the diagrams in $\mcD_n$ using the strict total ordering on orbits in Definition~\ref{defn:strict-total-ordering} and the arbitrary choice of ordering we've made within each orbit. Diagrams from each orbit  $\cO(a,b,c,d)$ are grouped together in the strict total ordering. 

We use this order for the rows in the Gram matrix, going from the smallest to the greatest $n$-diagram relative to the order. To label columns, we use involution $\iota$. Thus, if the $i$-row is labelled by diagram $X$, then $i$-th column is labelled by diagram $\iota(X)$. We denote this permuted Gram matrix by $G_{\iota}$. It can also be thought of as a block matrix, with blocks enumerated by pairs of orbits of $S_n$ on $\mcD_n$.

\begin{prop}\label{prop_orbits} Suppose $\cO>\cO'$ for orbits $\cO,\cO'\in \Orb_n$. Then  $\alpha(\overline{\iota(Y)}X)=0$ for any $X\in \cO, Y\in \cO'$.
\end{prop}

\begin{proof}  Let  $\cO= \cO(a,b,c,d)$ and $\cO'=\cO(a',b',c',d')$. Then $\iota(Y)\in\cO(b',a',d',c')$.  Condition $\cO>\cO'$ says that one of the following three conditions hold: 
\begin{itemize} 
\item $\ddef(X)>\ddef(Y)$, or $b+d> b'+d'$. Diagram $\iota(Y)$ has $a'+c'$ dots, and the coupling diagram  $\overline(\iota(Y))X$  has  $b+d+a'+c' > b'+d'+c'+a'=\comp(Y)$  dots, making it evaluate to zero, since one of the components has at least two dots. 

\item $\ddef(X)=\ddef(Y)$ and $\comp(X)<\comp(Y)$. This means $b+d=b'+d'$ and $a+b+c+d<a'+b'+c'+d'$. We can rewrite the inequality as $a+c<a'+c'$ (or that $X$ has fewer arcs than $Y$). 
Diagram $\iota(Y)$ has $a'+c'$ dots and the coupling $\overline{\iota(Y)}X$ has exactly $b+d+a'+c'=a'+b'+c'+d'=\comp(Y)$ dots.
In the coupling diagram $\overline{\iota(Y)}X$ we can slide the dots away from components of $\iota(Y)$ onto adjacent components of $X$. Since $\comp(X)<\comp(Y)$ and there are $\comp(Y)$ dots in the diagram, one of the components of $X$ will carry more than one dot. 
Consequently, the diagram $\overline{\iota(Y)}X$ will evaluate to $0$ via $\alpha$. 

\item  $\ddef(X)=\ddef(Y),$ $\comp(X)=\comp(Y)$ and $d>d'$. Repeating the arguments from the previous case, the diagram $X\cdot \iota(Y)$ has exactly $\comp(X)=\comp(Y)$ dots and they are distributed so that there's one dot on each component. In the coupling arcs of $X$ must pair to (reflected) arcs of $Y$ and cups of $X$ to (reflected) cups of $Y$, for otherwise $X\cdot \iota(Y)$  has less than $\comp(X)=\comp(Y)$ connected components and evaluates to $0$. Diagram $\overline{\iota(Y)}X$ consists of $c+d=c'+d'$ circles and $a+b=a'+b'$ arcs. The number of dots on cups of $X$ and $Y$ is $c'+d>c'+d'$, so it's greater than the number of circles of $\overline{\iota(Y)}X$, and in the latter coupling one of the circles carries more than one dot, evaluating to $0$. 
\end{itemize}
\end{proof}

\begin{prop} \label{prop_same_orbit}
Suppose that $X$ and $Y$ belong to the same orbit $\cO(a,b,c,d)$. Then 
\begin{equation}
    \alpha(\overline{\iota(Y)}X)=
    \begin{cases}
       \beta_1^{a+b}\gamma_1^{c+d} &  \textrm{ if } X=Y, \\
       0   & \textrm{ otherwise.}
    \end{cases}
\end{equation}
\end{prop}

\begin{proof} The number of dots on $\overline{\iota(Y)}X$ is $a+b+c+d=\comp(X)=\comp(Y)$. If arcs and cups of $X$ don't perfectly match those of $Y$ (which happens if and only if $X=Y$), the coupled diagram $\overline{\iota(Y)}X$ will have fewer than $\comp(X)$ components and evaluate to zero, due to the presence of $\comp(X)$ dots on it. The diagonal case $X=Y$ is clear, since $\overline{\iota(X)}X$ contains $a+b$ single-dotted intervals and $c+d$ single-dotted circles. 
\end{proof} 

\begin{theorem}
 Twisted Gram matrix $G_{\iota}$ is lower-triangular.
\end{theorem}

\begin{proof} Consider $n$-diagrams $X,Y$ with $X<Y$. They belong to orbits $\cO,\cO'$, respectively, with $\cO\le \cO'$. Let $i$ and $j$ be the rows corresponding to $X$ and $Y$, respectively. The $(i,j)$ entry  of $G_{\iota}$ is $\alpha(\overline{\iota(Y)}X)$. 

If $\cO<\cO'$ then  $\alpha(\overline{\iota(Y)}X)=0$ by Proposition~\ref{prop_orbits}. If $\cO=\cO'$ then $X\not= Y$ belong to same orbit and $\alpha(\overline{\iota(Y)}X)=0$ by Proposition~\ref{prop_same_orbit}. 
\end{proof} 

\begin{corollary} For a diagram $X\in\cO(a,b,c,d)$ the 
diagonal entry is 
\begin{equation}
    \alpha(\iota(X)X) = \gamma_{1}^{a+b}\beta_{1}^{c+d}.
\end{equation}
\end{corollary}

\begin{proof}
This is part of Proposition~\ref{prop_same_orbit}.
\end{proof}

\begin{figure}
    \centering
\begin{tikzpicture}[scale=1.1]
\draw[thick] (-1,0) -- (7,0);
\draw[thick] (-1,1) -- (7,1);
\draw[thick] (-1,1) -- (7,1);
\draw[thick,dotted] (1,2) -- (7,2);
\draw[thick] (-1,3) -- (7,3);
\draw[thick] (-1,4) -- (7,4);
\draw[thick] (-1,5) -- (7,5);
\draw[thick] (1,-1) -- (1,6);
\draw[thick] (2,-1) -- (2,6);
\draw[thick] (3,-1) -- (3,6);
\draw[thick,dotted] (4,-1) -- (4,5);
\draw[thick] (5,-1) -- (5,6);
\draw[thick] (6,-1) -- (6,6);
\draw[thick] (7,-1) -- (7,6);

\draw[thick] (1.3,5.25) -- (1.3,5.75);
\draw[thick] (1.7,5.25) -- (1.7,5.75);

\draw[thick] (2.3,5.75) .. controls (2.3,5) and (2.7,5) .. (2.7,5.75);

\node at (3.2,5.5) {$\Big\{$};
\draw[thick] (3.4,5.25) -- (3.4,5.75);
\draw[thick] (3.7,5.25) -- (3.7,5.75);
\draw[thick,fill] (3.8,5.5) arc (0:360:1mm);
\node at (4,5.25) {$,$};
\draw[thick] (4.3,5.25) -- (4.3,5.75);
\draw[thick,fill] (4.4,5.5) arc (0:360:1mm);
\draw[thick] (4.6,5.25) -- (4.6,5.75);
\node at (4.8,5.5) {$\Big\}$};

\draw[thick] (5.3,5.75) .. controls (5.3,5) and (5.7,5) .. (5.7,5.75);
\draw[thick, fill] (5.6,5.25) arc (0:360:1mm);

\draw[thick] (6.3,5.25) -- (6.3,5.75);
\draw[thick, fill] (6.4,5.5) arc (0:360:1mm);
\draw[thick] (6.7,5.25) -- (6.7,5.75);
\draw[thick, fill] (6.8,5.5) arc (0:360:1mm);

\draw[thick] (-.2,4.25) -- (-.2,4.75);
\draw[thick,fill] (-.1,4.5) arc (0:360:1mm);
\draw[thick] (.2,4.25) -- (.2,4.75);
\draw[thick,fill] (.3,4.5) arc (0:360:1mm);

\draw[thick] (-.2,3.75) .. controls (-.2,3) and (.2,3) .. (.2,3.75);
\draw[thick,fill] (.1,3.25) arc (0:360:1mm);

\node at (-.8,2) {$\Big\{$};
\draw[thick] (-.6,1.75) -- (-.6,2.25);
\draw[thick,fill] (-.5,2) arc (0:360:1mm);
\draw[thick] (-.3,1.75) -- (-.3,2.25); 
\node at (0,1.75) {$,$};
\draw[thick] (.3,1.75) -- (.3,2.25);
\draw[thick] (.6,1.75) -- (.6,2.25);
\draw[thick,fill] (.7,2) arc (0:360:1mm);
\node at (.8,2) {$\Big\}$};

\draw[thick] (-.2,.75) .. controls (-.2,0) and (.2,0) .. (.2,.75);

\draw[thick] (-.2,-.25) -- (-.2,-.75);
\draw[thick] (.2,-.25) -- (.2,-.75);

\node at (1.5,4.5) {\Large $\gamma_1^2$};
\node at (2.5,3.5) {\Large $\beta_1$};
\node at (3.5,2.5) {\Large $\gamma_1^2$};
\node at (4.5,1.5) {\Large $\gamma_1^2$};
\node at (5.5,0.5) {\Large $\beta_1$};
\node at (6.5,-.5) {\Large $\gamma_1^2$};
\node at (2.5,4.5) {\Large $0$};
\node at (3.5,4.5) {\Large $0$};
\node at (4.5,4.5) {\Large $0$};
\node at (5.5,4.5) {\Large $0$};
\node at (6.5,4.5) {\Large $0$};

\node at (3.5,3.5) {\Large $0$};
\node at (4.5,3.5) {\Large $0$};
\node at (5.5,3.5) {\Large $0$};
\node at (6.5,3.5) {\Large $0$};

\node at (4.5,2.5) {\Large $0$};
\node at (5.5,2.5) {\Large $0$};
\node at (6.5,2.5) {\Large $0$};

\node at (5.5,1.5) {\Large $0$};
\node at (6.5,1.5) {\Large $0$};

\node at (6.5,0.5) {\Large $0$};

\node at (1.5,3.5) {\Large $\gamma_1^2$};
\node at (1.5,2.5) {\Large $\gamma_0\gamma_1$};
\node at (1.5,1.5) {\Large $\gamma_0\gamma_1$};
\node at (1.5,0.5) {\Large $\gamma_0$};
\node at (1.5,-.5) {\Large $\gamma_0^2$};

\node at (2.5,2.5) {\Large $\gamma_1$};
\node at (2.5,1.5) {\Large $\gamma_1$};
\node at (2.5,0.5) {\Large $\beta_0$};
\node at (2.5,-.5) {\Large $\gamma_0$};

\node at (3.5,1.5) {\Large $0$};
\node at (3.5,0.5) {\Large $\gamma_1$};
\node at (3.5,-.5) {\Large $\gamma_0\gamma_1$};

\node at (4.5,0.5) {\Large $\gamma_1$};
\node at (4.5,-.5) {\Large $\gamma_0\gamma_1$};

\node at (5.5,-.5) {\Large $\gamma_1$};
\end{tikzpicture}
    \caption{Twisted Gram matrix for $n=2$. The determinant  is $\beta_1^2\gamma_1^8$.}
    \label{fig3_1}
\end{figure}

An example of the twisted Gram matrix is shown in Figure~\ref{fig3_1} for $n=2$. In this  case only one orbit has more than one element, and there are two choices for the total order. Each of them produces an upper-triangular matrix. Diagrams $X$ and $\iota(X)$ are displayed in the corresponding rows and columns. Table~\ref{table3_2} shows how the rows and columns are ordered. 
 
\begin{table}[ht]
\centering
\begin{tabular}{|c|c|c||c|c|c|}
\hline
\multicolumn{3}{c}{Rows} & 
\multicolumn{3}{c}{Columns} \\ \hline 
$(a,b,c,d)$ & $\ddef$ & $\comp$ & $(b,a,d,c)$ & $\ddef$ & $\comp$\\ 
\hline
$(0,2,0,0)$ & $2$ & $2$ & $(2,0,0,0)$ & $0$ & $2$ \\ 
\hline
$(0,0,0,1)$ & $1$ & $1$ & $(0,0,1,0)$ & $0$ & $1$\\ 
\hline
$(1,1,0,0)$ & $1$ & $2$ & $(1,1,0,0)$ & $1$ & $2$ \\ 
\hline
$(0,0,1,0)$ & $0$ & $1$ & $(0,0,0,1)$ & $1$ & $1$ \\ 
\hline
$(2,0,0,0)$ & $0$ & $2$ & $(0,2,0,0)$ & $2$ & $2$ \\ 
\hline
\end{tabular}
    \caption{Ordering of the orbits for the twisted Gram matrix for $n=2$.}
    \label{table3_2}
\end{table}

\begin{proposition}
Let $\vert \mathcal{O}(a,b,c,d) \vert $ denote the cardinality of the orbit $\mathcal{O}(a,b,c,d)$. Then
\[ 
\mathlarger{\left|  \mathcal{O}(a,b,c,d) \right| = \frac{n!}{a!\: b!\: c!\: d! \: 2^{c+d}}}.
\]
\end{proposition}

\begin{proof}
The cardinality of the orbit is the order of $S_n$ divided by the order of the stabilizer of any representative of the orbit. So
$\left| \mathcal{O}(a,b,c,d) \right| = \displaystyle{\frac{\vert S_n \vert }{\vert S_{nx} \vert}}$. Given a representative of $\mathcal{O}(a,b,c,d)$, we can can permute any of the un-dotted strands, which contributed a factor of $a!$, dotted strands, which contributes a factor of $b!$, un-dotted cups, which contributes a factor of $c!$, or dotted cups, which contributes a factor of $d!$. Moreover we can swap the two endpoints of any cup, which contributes a factor of $2^{c+d}$.
\end{proof} 

The determinant of the diagonal block for the orbit $\mathcal{O}(a,b,c,d)$ is 
\[ \left( \gamma_{1}^{a+b} \beta_{1}^{c+d} \right) ^{\left| \mathcal{O}(a,b,c,d) \right|}
\]
Taking the product over all orbits gives the determinant of the twisted Gram matrix:  
\[ \det G_{\iota} = \prod_{A+2B=n} \left( \gamma_{1}^A \beta_{1}^B \right)^{\displaystyle{\sum_{\stackrel{a+b=A,}{c+d=B}}} \vert \mathcal{O}(a,b,c,d) \vert}
\]
This determinant may differ by a sign from the  determinant of the original Gram matrix, $\det G_{\iota}=\pm \det G$. 

\begin{prop} For $n\ge 2$ the set $\mcD_n$ of $n$-diagrams constitutes a basis of the state space $A(n)$ if and only if $\beta_1\gamma_1 \not=0 $. 
\end{prop} 
\begin{proof}
This is immediate since the determinant is nonzero exactly when $\beta_1\gamma_1\not=0$. For $n=1$, the determinant is $-\beta_1^2$, and the corresponding condition is $\beta_1\not= 0$. \end{proof}

This proposition implies that, when $\beta_1\gamma_1\not= 0$, there is no difference between the skein category for this evaluation $\alpha$, as defined in~\cite{Kho20-1d-cob}, and the quotient $\Cob_{\alpha}$ of that category. In both categories, the defining relations are the evaluations of closed diagrams with one  or no dots via $\alpha$  and the  relation that two dots on a strand equal $0$. 

Furthermore, in these categories, spaces of homs from $n$ to $m$ can be naturally identified with $A(n+m)$ by moving the $n$ points from the source to the target of the cobordism via $n$ arcs. Consequently, the above proposition gives  bases of homs between all pairs of objects $(n,m)$ in $\Cob_{\alpha}$. Composition of basic morphisms is computed by concatenating them, evaluating closed components via $\alpha$, and further evaluating to $0$ if one of the remaining components contains more than one dot.  

When $\beta_1=\gamma_1=0$, a dot evaluates to $0$, the resulting categories have smaller hom spaces and relate to rook Brauer algebras~\cite{HdM}. It may also be interesting to investigate the case when only one of $\beta_1,\gamma_1$ equals zero.

\begin{figure}
    \centering
\begin{tikzpicture}[scale=1.2]
\draw[thick] ( 0,9) -- (9,9);
\draw[thick] (-1,8) -- (9,8);
\draw[thick] (-1,7) -- (9,7);
\draw[thick] (-1,6) -- (9,6);
\draw[thick] (-1,5) -- (9,5);
\draw[thick] (-1,4) -- (9,4);
\draw[thick] (-1,3) -- (9,3);
\draw[thick] (-1,2) -- (9,2);
\draw[thick] (-1,1) -- (9,1);
\draw[thick] (-1,0) -- (9,0);

\draw[thick] (-1,10) -- (1,8);
\node at (.35,8.25) {\Large $x$};
\node at (.67,8.75) {\large $\iota(x)$};
\node at (-.6,8.75) {\Large orbit};
\node at (-.6,8.35) {\Large size};
\node at (.4,9.75) {\Large orbit};
\node at (.4,9.35) {\Large size};

\draw[thick] (0,0) -- (0,9);
\draw[thick] (1,0) -- (1,10);
\draw[thick] (2,0) -- (2,10);
\draw[thick] (3,0) -- (3,10);
\draw[thick] (4,0) -- (4,10);
\draw[thick] (5,0) -- (5,10);
\draw[thick] (6,0) -- (6,10);
\draw[thick] (7,0) -- (7,10);
\draw[thick] (8,0) -- (8,10);
\draw[thick] (9,0) -- (9,10);

\node at (-0.5,7.5) {\Large $1$};
\node at (-0.5,6.5) {\Large $3$};
\node at (-0.5,5.5) {\Large $3$};
\node at (-0.5,4.5) {\Large $3$};
\node at (-0.5,3.5) {\Large $3$};
\node at (-0.5,2.5) {\Large $3$};
\node at (-0.5,1.5) {\Large $1$};
\node at (-0.5,0.5) {\Large $1$};



\draw[thick] (1.3,8.3) -- (1.3,8.7);
\draw[thick] (1.5,8.3) -- (1.5,8.7);
\draw[thick] (1.7,8.3) -- (1.7,8.7);
\node at (1.5,9.5) {\Large $1$};

\draw[thick] (2.3,8.3) -- (2.3,8.7);
\draw[thick] (2.5,8.3) -- (2.5,8.7);
\draw[thick] (2.7,8.3) -- (2.7,8.7);
\draw[thick,fill] (2.76,8.5) arc (0:360:0.55mm); 
\node at (2.5,9.5) {\Large $3$};

\draw[thick] (3.3,8.7) .. controls (3.33,8.15) and (3.47,8.15) .. (3.5,8.7);
\draw[thick] (3.7,8.3) -- (3.7,8.7);
\node at (3.5,9.5) {\Large $3$};

\draw[thick] (4.3,8.7) .. controls (4.33,8.15) and (4.47,8.15) .. (4.5,8.7);
\draw[thick] (4.7,8.3) -- (4.7,8.7);
\draw[thick,fill] (4.76,8.5) arc (0:360:0.55mm);
\node at (4.5,9.5) {\Large $3$};

\draw[thick] (5.3,8.7) .. controls (5.33,8.15) and (5.47,8.15) .. (5.5,8.7);
\draw[thick] (5.7,8.3) -- (5.7,8.7);
\draw[thick,fill] (5.45,8.3) arc (0:360:0.55mm);
\node at (5.5,9.5) {\Large $3$};

\draw[thick] (6.3,8.3) -- (6.3,8.7);
\draw[thick] (6.5,8.3) -- (6.5,8.7);
\draw[thick] (6.7,8.3) -- (6.7,8.7);
\draw[thick,fill] (6.56,8.5) arc (0:360:0.55mm);
\draw[thick,fill] (6.76,8.5) arc (0:360:0.55mm);
\node at (6.5,9.5) {\Large $3$};

\draw[thick] (7.3,8.7) .. controls (7.33,8.15) and (7.47,8.15) .. (7.5,8.7);
\draw[thick] (7.7,8.3) -- (7.7,8.7);
\draw[thick,fill] (7.45,8.3) arc (0:360:0.55mm);
\draw[thick,fill] (7.76,8.5) arc (0:360:0.55mm);
\node at (7.5,9.5) {\Large $3$};

\draw[thick] (8.3,8.3) -- (8.3,8.7);
\draw[thick] (8.5,8.3) -- (8.5,8.7);
\draw[thick] (8.7,8.3) -- (8.7,8.7);
\draw[thick,fill] (8.36,8.5) arc (0:360:0.55mm);
\draw[thick,fill] (8.56,8.5) arc (0:360:0.55mm);
\draw[thick,fill] (8.76,8.5) arc (0:360:0.55mm);
\node at (8.5,9.5) {\Large $1$};

\draw[thick] (.3,7.3) -- (.3,7.7);
\draw[thick] (.5,7.3) -- (.5,7.7);
\draw[thick] (.7,7.3) -- (.7,7.7);
\draw[thick,fill] (.36,7.5) arc (0:360:0.55mm); 
\draw[thick,fill] (.56,7.5) arc (0:360:0.55mm); 
\draw[thick,fill] (.76,7.5) arc (0:360:0.55mm); 

\draw[thick] (.3,6.3) -- (.3,6.7);
\draw[thick] (.5,6.3) -- (.5,6.7);
\draw[thick] (.7,6.3) -- (.7,6.7);
\draw[thick,fill] (.36,6.5) arc (0:360:0.55mm); 
\draw[thick,fill] (.56,6.5) arc (0:360:0.55mm);

\draw[thick] (.3,5.7) .. controls (.33,5.15) and (.47,5.15) .. (.5,5.7);
\draw[thick] (.7,5.3) -- (.7,5.7);
\draw[thick,fill] (.45,5.3) arc (0:360:0.55mm);
\draw[thick,fill] (.76,5.5) arc (0:360:0.55mm);

\draw[thick] (.3,4.7) .. controls (.33,4.15) and (.47,4.15) .. (.5,4.7);
\draw[thick] (.7,4.3) -- (.7,4.7);
\draw[thick,fill] (.46,4.3) arc (0:360:0.55mm); 

\draw[thick] (.3,3.7) .. controls (.33,3.15) and (.47,3.15) .. (.5,3.7);
\draw[thick] (.7,3.3) -- (.7,3.7);
\draw[thick,fill] (.76,3.5) arc (0:360:0.55mm);

\draw[thick] (.3,2.3) -- (.3,2.7);
\draw[thick] (.5,2.3) -- (.5,2.7);
\draw[thick] (.7,2.3) -- (.7,2.7);
\draw[thick,fill] (.36,2.5) arc (0:360:0.55mm);

\draw[thick] (.3,1.7) .. controls (.33,1.15) and (.47,1.15) .. (.5,1.7);
\draw[thick] (.7,1.3) -- (.7,1.7);
 
\draw[thick] (.3,.3) -- (.3,.7);
\draw[thick] (.5,.3) -- (.5,.7);
\draw[thick] (.7,.3) -- (.7,.7);

\node at (2.5,7.5) {\Large $0$};
\node at (3.5,7.5) {\Large $0$};
\node at (4.5,7.5) {\Large $0$};
\node at (5.5,7.5) {\Large $0$};
\node at (6.5,7.5) {\Large $0$};
\node at (7.5,7.5) {\Large $0$};
\node at (8.5,7.5) {\Large $0$};
\node at (3.5,6.5) {\Large $0$};
\node at (4.5,6.5) {\Large $0$};
\node at (5.5,6.5) {\Large $0$};
\node at (6.5,6.5) {\Large $0$};
\node at (7.5,6.5) {\Large $0$};
\node at (8.5,6.5) {\Large $0$};
\node at (4.5,5.5) {\Large $0$};
\node at (5.5,5.5) {\Large $0$};
\node at (6.5,5.5) {\Large $0$};
\node at (7.5,5.5) {\Large $0$};
\node at (8.5,5.5) {\Large $0$};
\node at (5.5,4.5) {\Large $0$};
\node at (6.5,4.5) {\Large $0$};
\node at (7.5,4.5) {\Large $0$};
\node at (8.5,4.5) {\Large $0$};
\node at (6.5,3.5) {\Large $0$};
\node at (7.5,3.5) {\Large $0$};
\node at (8.5,3.5) {\Large $0$};
\node at (7.5,2.5) {\Large $0$};
\node at (8.5,2.5) {\Large $0$};
\node at (8.5,1.5) {\Large $0$};

\node at (1.5,7.5) {\large $\beta_1^3$ \hspace{-1mm}{\Large $\I_1$}};
\node at (1.5,6.5) {\Large $*$};
\node at (1.5,5.5) {\Large $*$};
\node at (1.5,4.5) {\Large $*$};
\node at (1.5,3.5) {\Large $*$};
\node at (1.5,2.5) {\Large $*$};
\node at (1.5,1.5) {\Large $*$};
\node at (1.5,0.5) {\Large $*$};
\node at (2.5,6.5) {\large $\beta_1^3$ \hspace{-1mm}{\Large $\I_3$}};
\node at (2.5,5.5) {\Large $0$};
\node at (2.5,4.5) {\Large $*$};
\node at (2.5,3.5) {\Large $*$};
\node at (2.5,2.5) {\Large $*$};
\node at (2.5,1.5) {\Large $*$};
\node at (2.5,0.5) {\Large $*$};
\node at (3.5,5.5) {\large $\beta_1\gamma_1$ \hspace{-1mm}{\Large $\I_3$}};
\node at (3.5,4.5) {\Large $*$};
\node at (3.5,3.5) {\Large $*$};
\node at (3.5,2.5) {\Large $*$};
\node at (3.5,1.5) {\Large $*$};
\node at (3.5,0.5) {\Large $*$};
\node at (4.5,4.5) {\large $\beta_1\gamma_1$ \hspace{-1mm}{\Large $\I_3$}};
\node at (4.5,3.5) {\Large $0$};
\node at (4.5,2.5) {\Large $*$};
\node at (4.5,1.5) {\Large $*$};
\node at (4.5,0.5) {\Large $*$};
\node at (5.5,3.5) {\large $\beta_1\gamma_1$ \hspace{-1mm}{\Large $\I_3$}};
\node at (5.5,2.5) {\Large $*$};
\node at (5.5,1.5) {\Large $*$};
\node at (5.5,0.5) {\Large $*$};
\node at (6.5,2.5) {\large $\beta_1^3$ \hspace{-1mm}{\Large $\I_3$}};
\node at (6.5,1.5) {\Large $*$};
\node at (6.5,0.5) {\Large $*$};
\node at (7.5,1.5) {\large $\beta_1\gamma_1$ \hspace{-1mm}{\Large $\I_3$}};
\node at (7.5,0.5) {\Large $*$};
\node at (8.5,0.5) {\large $\beta_1^3$ \hspace{-1mm}{\Large $\I_1$}};
\end{tikzpicture}
    \caption{Twisted Gram block matrix when $n=3$, where rows represent orbits, and the columns dual orbits. One element of the orbit is shown in each row and column. 
    The cardinality of the orbit is shown to the left or above each orbit representative. $\I_k$ denotes the $k\times k$ identity matrix. The determinant is $\beta_1^{36} \gamma_1^{12}$.}
    \label{fig3_2}
\end{figure}

\begin{table}[ht]
\centering
\begin{tabular}{|c|c|c||c|c|c|}
\hline
\multicolumn{3}{c}{Along the row} & 
\multicolumn{3}{c}{Along the column} \\ \hline 
$(a,b,c,d)$ & $\ddef$ & $\comp$ & $(b,a,d,c)$ & $\ddef$ & $\comp$\\ 
\hline
$(0,3,0,0)$ & $3$ & $3$ & $(3,0,0,0)$ & $0$ & $3$ \\ 
\hline
$(0,1,0,1)$ & $2$ & $2$ & $(1,0,1,0)$ & $0$ & $2$\\ 
\hline
$(1,2,0,0)$ & $2$ & $3$ & $(2,1,0,0)$ & $1$ & $3$ \\ 
\hline
$(1,0,0,1)$ & $1$ & $2$ & $(0,1,1,0)$ & $1$ & $2$ \\ 
\hline
$(0,1,1,0)$ & $1$ & $2$ & $(1,0,0,1)$ & $1$ & $2$ \\ 
\hline
$(2,1,0,0)$ & $1$ & $3$ & $(1,2,0,0)$ & $2$ & $3$ \\ 
\hline
$(1,0,1,0)$ & $0$ & $2$ & $(0,1,0,1)$ & $2$ & $2$ \\ 
\hline
$(3,0,0,0)$ & $0$ & $3$ & $(0,3,0,0)$ & $3$ & $3$ \\ 
\hline
\end{tabular}
    \caption{Ordering of the orbits for the Gram matrix for $n=3$.}
    \label{table3_3}
\end{table}

\input{fig3_3} 

\begin{table}[ht]
\centering
\begin{tabular}{|c|c|c||c|c|c|}
\hline
\multicolumn{3}{c}{Along the row} & 
\multicolumn{3}{c}{Along the column} \\ \hline 
$(a,b,c,d)$ & $\ddef$ & $\comp$ & $(b,a,d,c)$ & $\ddef$ & $\comp$\\ 
\hline
$(0,4,0,0)$ & $4$ & $4$ & $(4,0,0,0)$ & $0$ & $4$ \\ 
\hline
$(0,2,0,1)$ & $3$ & $3$ & $(2,0,1,0)$ & $0$ & $3$\\ 
\hline
$(1,3,0,0)$ & $3$ & $4$ & $(3,1,0,0)$ & $1$ & $4$ \\ 
\hline
$(0,0,0,2)$ & $2$ & $2$ & $(0,0,2,0)$ & $0$ & $2$ \\ 
\hline
$(0,2,1,0)$ & $2$ & $3$ & $(2,0,0,1)$ & $1$ & $3$ \\ 
\hline
$(1,1,0,1)$ & $2$ & $3$ & $(1,1,1,0)$ & $1$ & $3$ \\ 
\hline
$(2,2,0,0)$ & $2$ & $4$ & $(2,2,0,0)$ & $2$ & $4$ \\ 
\hline
$(0,0,1,1)$ & $1$ & $2$ & $(0,0,1,1)$ & $1$ & $2$ \\ 
\hline
$(1,1,1,0)$ & $1$ & $3$ & $(1,1,0,1)$ & $2$ & $3$ \\ 
\hline
$(2,0,0,1)$ & $1$ & $3$ & $(0,2,1,0)$ & $2$ & $3$ \\ 
\hline
$(3,1,0,0)$ & $1$ & $4$ & $(1,3,0,0)$ & $3$ & $4$ \\ 
\hline
$(0,0,2,0)$ & $0$ & $2$ & $(0,0,0,2)$ & $2$ & $2$ \\ 
\hline
$(2,0,1,0)$ & $0$ & $3$ & $(0,2,0,1)$ & $3$ & $3$ \\ 
\hline
$(4,0,0,0)$ & $0$ & $4$ & $(0,4,0,0)$ & $4$ & $4$ \\ 
\hline
\end{tabular}
    \caption{Ordering of the orbits for the Gram matrix for $n=4$.}
    \label{table3_4}
\end{table}

\clearpage   
  
\section*{Acknowledgments}
The authors would like to thank Mikhail Khovanov for introducing the authors to decorated low-dimensional cobordisms, and for guidance and feedback on this manuscript.

\bibliographystyle{amsalpha} 

\bibliography{1D-defect-det}

\end{document}